\documentclass[11pt]{article}

\usepackage{amsmath,amssymb,amsthm}
\usepackage{epsfig, color}
\usepackage[ruled,vlined]{algorithm2e}

\newtheorem{theorem}{Theorem}
\newtheorem{lemma}{Lemma}

\newcommand{\Reals}{{\mathbb R}}

\newcommand{\prob}{{\mathbb P}}

\begin{document}
%

\title{Probabilistic consensus via polling and majority rules}

\author{James Cruise and Ayalvadi Ganesh}

\maketitle

\begin{abstract}
In this paper, we consider lightweight decentralised algorithms for
achieving consensus in distributed systems. Each member of a
distributed group has a private value from a fixed set consisting of,
say, two elements, and the goal is for all members to reach consensus
on the majority value. We explore variants of
the voter model applied to this problem. In the voter model, each node
polls a randomly chosen group member and adopts its value. The process
is repeated until consensus is reached. We generalize this so that
each member polls a (deterministic or random) number of other group
members and changes opinion only if a suitably defined super-majority
has a different opinion. We show that this modification greatly speeds
up the convergence of the algorithm, as well as substantially reducing
the probability of it reaching consensus on the incorrect value. 
\end{abstract}

\section{Introduction} \label{sec:intro}

The consensus problem is that of getting all agents in a distributed system to agree on a value from an initial set of values that they each possess. We consider the special case where this initial set consists of just two values, 0 and 1. Whereas the traditional consensus problem only requires agreement, we shall also require that the agreed value coincide with the initial majority value. Our focus is on simple and lightweight decentralized algorithms which offer probabilistic performance guarantees, and which achieve fast convergence to the consensus value.

There are two strands of motivation for this work. The first comes from a variety of co-ordination problems in distributed systems where the main goal is to achieve agreement and the value agreed upon is only of secondary importance. The second motivation comes from social learning, where the interest is in whether agents with limited private information can aggregate their information via simple interaction rules to learn the ``true state of nature". This phenomenon has gained considerable popularity recently under the name of \emph{the wisdom of crowds} \cite{crowds}. It is also of interest in modelling the spread of opinions and influence in social networks, both online and offline, and the spread and adoption of competing ideas or technologies.

In the standard voter model, which is one of the earliest models of consensus, each agent initiates contacts at unit rate, i.e., at the points of a Poisson process of rate 1, choosing the contact uniformly at random in each instance. Contacts are chosen independent of past choices, and the choices of other agents. The agent initiating the contact then simply adopts the opinion of the agent contacted. This is a \emph{pull} model, though the \emph{push} version is exactly analogous. The number of agents having each value evolves as a Markov chain. In fact, the number of agents having value 1, say, is a martingale. Hence, the probability of reaching consensus on the value 1 is simply the fraction of nodes which initially had value 1. In particular, if this initial fraction is smaller than a half, then the probability of reaching the wrong consensus, i.e., agreeing on the minority value, is a constant that does not depend on the number of agents but only on the initial fraction. Moreover, it is known that the expected time to reach consensus grows linearly in the total number of agents. 

In a variant of this model, the agents live on the nodes of a graph and an agent can only contact its neighbours in the graph. Again, we consider a Markov model in which agent (or node) $i$ contacts node $j$ according to a Poisson process of rate $q_{ij}$, independent of the past and of other nodes. We assume that the rate matrix $Q=\{ q_{ij}, 1\le i,j\le n \}$, is irreducible, which is the case if the corresponding directed graph is strongly connected. Here $n$ is the total number of agents. Let $\pi$ denote the unique invariantdistribution of the rate matrix $Q$, i.e., the unique probability vector solving the equation $\pi Q=0$. Then it can be shown that $\pi \cdot X(t)$ is a martingale, where $X_i(t)$ denotes the value at node $i$ at time $t$, and $\pi \cdot X(t)$ denotes the inner product of the vectors $\pi$ and $X(t)$. (The analogous result for a discrete-time version of this model was established in \cite{peleg}, but the extension to continuous time is straightforward.) Consequently, the probability that the eventual consensus value is 1 is simply $\pi \cdot X(0)$. In this setting, agents have different amounts of influence as described by the invariant distribution $\pi$, but the results are otherwise quite similar to the case of the classical voter model. In particular, the probability of reaching the wrong consensus is substantial, and typically doesn't decay in the system size. 

\subsection{ Related work}

There has been a considerable body of recent work involving variants of the classical voter model, 
some of which has been motivated by the distributed consensus problem.
A simple 3-state variant of the voter model on the complete graph was studied in \cite{milan}, where the third state can be thought of as undecided. If an agent with value 0 polls one with value 1 or vice versa, then it moves to the undecided state. If an agent polls a node in the undecided state, then it doesn't change state. Finally, an undecided agent adopts the value of the agent it polls. It is shown in \cite{milan} that, for this model, the probability of reaching the wrong consensus decays exponentially in the number of agents, $n$, and the decay rate is the Kullback-Leibler divergence between the Bernoulli distribution with parameter equal to the initial fraction of agents with value 1, and the Bernoulli distribution with parameter one half. Moreover, the time to reach consensus is logarithmic in the number of nodes. One could ask whether having additional states further reduces the error probability or, alternatively, whether the same gain can be achieved if only a fraction of nodes have these extra states. Rather than answer these questions in the multiple-states setting, we continue to work with just two states but consider more general polling schemes. We show that polling two or more nodes rather than just one, and only changing state if a super-majority of the polled nodes have a different state, leads to error probabilities decaying exponentially in the number of agents, and in the time to reach consensus being logarithmic in the number of agents.

A discrete-time version of the majority voter model on infinite regular trees was studied in \cite{montanari}. 
In this model, nodes are initially assigned values of 0 or 1 according to independent and identically distributed (i.i.d.) 
Bernoulli random variables. This is the only randomness in the model. The subsequent evolution is deterministic; in 
each time step, each node updates its value to the majority value represented among its neighbours in the tree. 
The authors obtain bounds on the minimum initial bias (difference between probabilities of 0 and 1) required to 
guarantee that all nodes eventually reach consensus. Similar results are obtained in \cite{mossel} for finite regular 
graphs that are expanders. Specifically, the authors show that, for sufficiently large initial bias, consensus on the initial 
majority value is reached with probability tending to 1 as the size of the graph tends to infinity. A probabilistic version 
of the majority voter model, in which each agent randomly samples a subset of its neighbours in each time step and 
adopts the majority value in this subset, is studied in \cite{moaz}. Here, there is randomness in the initial condition 
as well as the sampling at each node and time step. The authors show that the model reaches consensus on the
initial majority value with high probability. The error probability bounds in their model are not as good as the ones in 
this paper, but they are applicable to a wide class of graphs, whereas our results are only for complete graphs.

There is a well known duality between the voter model and coalescing random walks, which can be 
exploited to obtain bounds on the time to reach consensus in the voter model. On the complete graph, it shows 
that the expected time to reach consensus scales linearly in the number of nodes. Recent work in \cite{cooper} 
has used this approach to derive bounds on the consensus time in general graphs in terms of the number of 
nodes, the variance of node degrees, and the spectral gap of the transition probability matrix for the random walk 
associated with the voter model. Unfortunately, the duality with random walks does not extend to the generalised 
voter model considered in this paper. There also doesn't seem to be a natural martingale associated with it; in the 
voter model, this martingale is intimately related to the dual random walk. Hence, we need a quite different approach 
for our model, which doesn't extend to general graphs.

Consensus can also be reached without using variants of voter models. Instead, agents could simply count the number of agents with values 0 and 1, and thereby choose the majority. It is shown in \cite{shah+aoyama} that such 
counting problems can be approximated using a gossip algorithm which involves propagating real numbers rather than 
values from a finite set. If the approximation error in counting is smaller than the margin between the 0 and 1 votes, then this leads to the correct decision regarding the initial majority value. The time required by this approximate counting algorithm is shown in \cite{shah+aoyama} to scale logarithmically in the number of nodes in complete graphs, whereas 
on general graphs, it can be bounded in terms of the expansion properties of the graph. 

In contrast, quantized gossip algorithms, which use more than two states and more complicated update rules, can be constructed to guarantee convergence to the correct value \cite{kashyap, benezit, alex+john}. Bounds on the convergence time of these algorithms have been obtained in \cite{milan2, shang}.

 Finally, we briefly mention the relevance of the work in this paper to distributed systems. We present 
results only for the complete graph, but even these can be relevant to distributed systems provided the update 
algorithms described here are applied to nodes selected uniformly at random from the population. Such sampling 
is possible by performing independent random walks on the network, one for each sample required. The number 
of steps required to obtain a nearly uniform random sample depends on the mixing time of the random walk, 
which is in turn related to the spectral gap of the Laplacian matrix of the graph. The key point is that if the graph 
is an expander, as many networks of practical interest are, then the random walk mixes quickly and it is typically 
possible to randomly sample a node with cost which is logarithmic in the number of nodes.

The rest of the paper is organized as follows. In the next section, we set out the variants of the voter model that we shall consider. We present an analysis of their performance, in terms of the time to reach consensus and the probability of reaching consensus on the wrong value, in Sections 3 and 4, before concluding in Section 5 with a discussion of directions for future work.

\section{Generalized voting algorithm} \label{sec:voting_algo}

We study voting algorithms on the complete graph on $N$ nodes. Each node has an initial state in $\{0,1\}$. 
We associate a unit rate exponential clock with each node, which is independent of its state and the states of 
the other nodes. At each tick of this clock, the node carries out an update of its current state according to an 
algorithm indexed by two parameters, $m$ and $d$. Here, $m$ represents the number of nodes the node polls 
in an update step, and $d$ is the minimum number of these that will need to disagree with its current state for it 
to change state. More precisely, the update algorithm, algorithm \ref{alg}, at each node $v$ is as follows: the 
node samples $m$ nodes $v_1, v_2,\ldots, v_m$ uniformly at random and finds out their current states. 
If $d$ or more have a current state different from its own, then $v$ changes its current state; otherwise, 
it retains its state.  For example, if $m=3$ and $d=2$, then each node samples 3 nodes and only changes 
state if a simple majority of them have a different state from its own. The classical voter model is recovered 
by taking $m=d=1$. For definiteness, and to simplify the analysis, we assume that these $m$ nodes are 
sampled with replacement. Sampling without replacement will not affect our results, which are of an asymptotic 
character for large populations.  
For large $N$ the probability of sampling a node twice in an update step is $O(\frac{1}{N})$. 
We can simulate sampling without replacement by first sampling with replacement, but rejecting the sample 
and resampling if a node is repeated. This coupling demonstrates that the proportion of update steps in which 
sampling with replacement yields a different sample than sampling without replacement tends to zero as we 
increase the number of nodes. Intuitively, this suggests that both sampling schemes should result in convergence 
to the same consensus value. We shall return to this point after the statement of Theorem \ref{hitprob}.

\begin{algorithm}
 \SetAlgoLined
 \KwData{$X$ state of each node}
 $i \gets$ sample ($1,n$) \tcc*[r]{select node to update uniformly at random}
 $\sigma\gets$ sample($m,N$) \tcc*[r]{sample $m$ nodes from the population with replacement}
 $count \gets \sum_{j\in \sigma} X[j]\; $ \tcc*[r]{count nodes in sample with state 1}
 \uIf{ $X[i]=0$ \& $count\geq d$}{
 $X[i] \gets 1$\;
 }
\ElseIf{$X[i]=1$ \& $m-count\geq d$}{
 $X[i] \gets 0$\;
 }

 \caption{$(m,d)$ update algorithm}\label{alg}
\end{algorithm}

If we define the state of the system to be the number of nodes in state 1, then it is easy to see that it evolves as a 
continuous time Markov chain on the finite state space $\{0,1,2,\ldots,N\}$, with the states $0$ and $N$ being absorbing. 
These are the only absorbing states and they are reachable from all other states, and so the Markov process will eventually end up in one of them. We are interested in the probability of the system ending in state $N$ (i.e., all nodes in state 1) if we started with a proportion $\alpha$ of the nodes initially having state $1$. We are also interested in the expected length of time it takes to reach a consensus, i.e., to reach an absorbing state. 

In the case of the classical voter model, it is known that the probability of being absorbed into the state $N$ is equal to the proportion of nodes which  start in state 1, and that the time to consensus scales linearly with the number of nodes, as remarked in the Introduction. Thus, the probability of reaching the wrong consensus does not decay with system size but only with how far the initial proportion is from an equal split. Moreover, the time to consensus is large. One of our motivations for introducing the generalized class of algorithms is to see if they can reduce the probability of converging to the wrong consensus as well as speed up the convergence. Another motivation is to explore the behaviour of a wider class of social learning mechanisms.

In the next section, we examine the performance of algorithms where the parameters $m$ and $d$ are fixed, and the same at every node. In the following section, we consider random $m$ and $d$ as a way of modelling heterogeneity among the nodes, which has strong motivation in a social learning context. The ideas and techniques we use for the deterministic case transfer across with minor changes. We present some conjectures and future directions for research in the final section.

\section{Analysis of the deterministic case} \label{analysis_det}
\subsection{Probability of reaching the incorrect consensus} \label{analysis_det_prob}

We consider the generalized voter model with parameters $(m,d)$ as described in the previous section. Let $X(t)$ denote the number of nodes in state $1$ at time $t$. Then $(X(t), t\in \Reals)$ evolves as a continuous time Markov chain on $\{0,1,2,\ldots,N\}$ with the transition rates
\begin{eqnarray*}
q_{n,n+1} &=& (N-n) \prob \Bigl( \mbox{Bin}\Bigl(m,\frac{n}{N}\Bigr) \ge d \Bigr), \\
q_{n,n-1} &=& n \prob \Bigl( \mbox{Bin}\Bigl(m,\frac{N-n}{N}\Bigr) \ge d \Bigr),
\end{eqnarray*}
for $0\le n\le N$. We are interested in computing the probability that the system ends in state $N$ given that we start with a proportion $\alpha$ of nodes in state $1$.

We denote this hitting probability by the function $h_N$, defined as follows.  For $\alpha = i/N$ with $i \in \{0,1,\ldots,N\}$, set
\begin{equation} \label{hndef_int}
h_N(\alpha)=\mathbb{P}(X(\infty) =N | X(0)= \lfloor \alpha N \rfloor),
\end{equation}
and extend the definition to $\alpha \in [0,1]$ by linear interpolation:
\begin{equation} \label{hndef_gen}
h_N(\alpha) = (N\alpha - \lfloor N \alpha \rfloor) h_N \Bigl( \frac{ \lceil \alpha N \rceil }{N} \Bigr)
+ (\lceil N \alpha \rceil - N\alpha) h_N \Bigl( \frac{ \lfloor \alpha N \rfloor }{N} \Bigr).
\end{equation}

We begin by considering the special case $d=m$, where a node only changes states if all $m$ nodes it samples have the opposite state. 
\begin{theorem}\label{mmprob}
Consider the $(m,m)$ algorithm with $m>1$ on the complete graph on $N$ nodes. Let $\{X(t)\}_{t\geq0}$ be the number of nodes whose state is 1 at time $t$ then we have
$$\mathbb{P}(X(\infty) =N | X(0)=i) =\sum_{k=0}^{i-1}  {N-1\choose k}^{m-1} \Bigm/ \sum_{k=0}^{N-1} {N-1\choose k}^{m-1}$$
Further for $\alpha \in (0,1/2)$ the 
probability of reaching consensus on 1 (the minority value initially) satisfies 
\begin{eqnarray*}
h_N(\alpha) &\le& c \exp \Bigl( -(N-1)(m-1) D\Bigl( \alpha; \frac{1}{2} \Bigr) \Bigr),
\end{eqnarray*}
where $c>0$ is a constant that does not depend on $N$, 
$D(p;1/2)=\log 2-H(p)$ is the relative entropy or Kullback-Leibler divergence
of the Bernoulli($p$) distribution with respect to the Bernoulli(1/2) distribution,
and $H(p)$ is the Shannon entropy of the Bernoulli($p$) distribution. Moreover,
the two sides of the above inequality are logarithmically equivalent as $N$ tends to
infinity. 
\end{theorem}

In words, the theorem says that the probability of reaching the incorrect consensus decays exponentially in the population size $N$, and the error exponent is a multiple of a Kullback-Leibler divergence involving the initial condition. If we compare the result in the theorem for $m=2$ to what is known for the classical voter model, we see that even just polling two nodes rather than a single node provides a substantial improvement: the probability of obtaining the wrong result decays exponentially in the number of nodes, which is not seen in the original voter model.

We have moved the proofs of all the theorems to the appendix but give a brief outline of the argument here. In order to compute hitting probabilities for the Markov chain $X(t)$, we use the well-known resistor analogue described by Doyle and  Snell \cite{resistor}, for example. Consider $N$ resistors $R_1,R_2,\ldots,R_N$ placed in series and connected to a 1 volt battery. Let $V_i$ denote the voltage at the junction between resistors $R_i$ and $R_{i+1}$ with $V_0=0$ and $V_N=1$. The voltages $V_i$ satisfy the relation
\begin{equation} \label{eq:voltage}
V_i = \frac{R_i V_{i+1} + R_{i+1} V_{i-1}}{R_i+R_{i+1}}
\end{equation}
for all $i$ between $1$ and $N-1$. On the other hand, by conditioning on the first step of the Markov chain $X(t)$ from state $i$, we can see that the hitting probabilities $h_N(i/N)$ satisfy the recursion
\begin{equation} \label{eq:hitting}
h_N \Bigl( \frac{i}{N} \Bigr) = p_{i,i+1} h_N \Bigl( \frac{i+1}{N} \Bigr) + p_{i,i-1} h_N \Bigl( \frac{i-1}{N} \Bigr),
\end{equation}
where $p_{i,i+1}$ and $p_{i,i-1}$ are the transition probabilities for the embedded discrete-time Markov chain associated with $X(t)$, obtained by embedding at the jump times. These probabilities are proportional to the corresponding transition rates. Moreover, $h_N(0)=0$ and $h_N(1)=1$. Thus, comparing equations (\ref{eq:voltage}) and (\ref{eq:hitting}), we see that they have the same solutions if we set the resistor ratios $R_{i+1}/R_i$ to be the same as the transition rates ratios $q_{i,i-1}/q_{i,i+1}$. When we do this, we find that the resistor values are proportional to the $(m-1)^{\rm th}$ power of the binomial coefficients ${N\choose i}$. This allows us to apply the Chernoff bound for the binomial distribution to obtain the bound on the hitting probabilities claimed in the theorem. Details can be found in the appendix.

We now proceed to consider the more general setting. 


\begin{theorem}\label{hitprob}
Consider the $(m,d)$ algorithm with $m\geq 2$ on the complete graph on $N$ nodes.
If $d>m/2$ we have for all $\alpha \in (0,1/2)$ that
$$\limsup_{N \rightarrow \infty} \frac{1}{N} \log(h_N(\alpha)) \leq - \int_{\alpha}^{\frac{1}{2}} \log g(x) dx, $$
with $g: [0,1] \rightarrow 
\mathbb{R}^+\cup +\infty$  and $g(x)>1$ for all $x \in[0,1/2)$. Spefically 
$$ g(x) = \frac{x \mathbb{P}( Z_x\leq(m-d))}{(1-x)\mathbb{P}(Z_x\geq d)}, $$ 
where, for $x\in [0,1]$, 
$Z_x$ denotes a random variable with the Bin$(m,x)$ distribution. 

\end{theorem} 

The error exponent of the algorithm is governed by the function $g$, which encodes the drift of the Markov chain; 
it gives the ratio of the probability of an up step to the probability of a down step as a function of the fraction of 
nodes in state $1$. It is readily verified that if $d=m$, then this theorem recovers a weaker form of the result in 
Theorem \ref{mmprob}.

Note that $g(x)>1$ for all $x\in[0,1/2)$ if $d> m/2$, i.e., if we require a strict majority to force a state change 
at a node. If we were to sample nodes without replacement, then the random variable $Z_x$ would have a 
hypergeometric distribution instead of a binomial distribution. In other words, we would have
$$
\mathbb{P}(Z_x=k) = \frac{ {\lfloor xN \rfloor \choose k}{\lceil (1-x)N \rceil \choose m-k}}{{N \choose m}}.
$$
As $N$ tends to infinity, and for any fixed $x\in (0,1)$ and $k\in \{ 0,1,\ldots,m \}$, the expression above tends 
to the probability that a Bin$(m,x)$ random variable takes the value $k$. This confirms what we noted intuitively 
earlier, that the difference between sampling with and without replacement is negligible.


\subsection{Time to reach consensus} \label{analysis_det_time}

In addition to the probability of error, we are also interested in the time to reach consensus. For an algorithm to be useful in practice, we need this time to be short. Recall that in the classical voter model, the time to consensus is $\Theta(N)$, i.e, it grows linearly in the number of nodes. Before stating results on the time to convergence for our class of algorithms, we define functions $\tau^\alpha_N :[\alpha,1-\alpha]\rightarrow \mathbb{R}^+$ by
$$\tau_N^\alpha(x) =\mathbb{E}_{\lfloor x N\rfloor }(\inf \{ t>0: X(t) \leq  \alpha N \textrm{ or } X(t)\geq (1-\alpha) N \}),$$
where the notation $\mathbb{E}_y$ denotes the expectation starting from the initial condition $y$. In words, $\tau_N^\alpha(x)$ is the expected time to come within a fraction $\alpha$ of consensus starting from an initial proportion $x$ of nodes in state 1, and $\tau_N^0(x)$ is the corresponding expected time to reach complete consensus. Our next theorem states that the time to reach consensus starting from any initial condition is $O(\log N)$, whereas the time to come within any non-zero fraction $\alpha$ of consensus is $O(1)$. To put it another way, the time to get arbitrarily close to consensus (as a proportion of all nodes) does not grow with system size, and the time to reach full consensus only grows logarithmically. The logarithmic term is unavoidable and comes from a lower bound on the time for the last node to contact any other node.

\begin{theorem}\label{hittime}
Consider the $(m,d)$ algorithm with $2d>m\geq d$ on the complete graph on $N$ nodes. For all $x\in[0,1/2)$ we have
$\tau^0_N(x)= \Theta (\log(N))$ and for $\alpha>0$ we have $\tau^\alpha_N(x)= O(1)$. 
\end{theorem}

The proof is in the appendix. The idea is to bound the number of visits to each state before reaching an absorbing state (or the boundary of the specified region for approximate consensus). These bounds can be obtained by constructing a simplified process whose consensus time stochastically dominates that of the original process. The key idea, familiar from the gambler's ruin problem, is that for a random walk with drift, the expected number of visits to each state is a constant that does not grow with $N$. (In the classical gambler's ruin problem, the drift is constant, whereas in our setting the drift is state-dependent, but this complication is manageable.) The number of states in the state space is proportional to $N$, but so are the transition rates of the Markov chain (since all $N$ nodes are independently polling other nodes at unit rate) except very close to the boundary. It is these boundary effects that make the time to reach consensus logarithmic rather than bounded. See the appendix for details.

\subsubsection{Example}

To better understand the results given previously we consider the specific example of the $(2,2)$ algorithm. Here a node upon waking contacts 2 nodes and only changes state if both nodes disagree with its current state. As before we let $X$ be the number of nodes in state 1 with a total of $N$ nodes. Previously we have seen that $X$ evolves as a continuous time Markov chain. Let $p^N(x)$ be the probability of obtaining a consensus of all nodes being in state 1 and $t^N(x)$ be the expected time to any consensus with $x$ being the initial number of nodes in state 1. For a given $N$ 
we can calculate $p^N(x)$ and $t^N(x)$ exactly by solving the following recurrence relations:
\begin{multline*}(x+(N-x))p^N(x) = xp^N(x+1)  + (N-x)p^N(x-1) \\ \textrm{ with } p^N(0)=0, \; p^N(N)=1,\end{multline*}
and
\begin{multline*}
(x+(N-x))t^N(x)= xt^N(x+1)  + (N-x)t^N(x-1) + \frac{N^2}{x(N-x)}\\ \textrm{ with } t^N(0)=0, \; t^N(N)=0. \end{multline*}
In figure \ref{fig::exact} we see a plot of $p^N$ and $t^N$ for $N=1000$ as we vary the initial number of nodes with state 1.
 
\begin{figure}[!t]
\centering
\epsfig{file=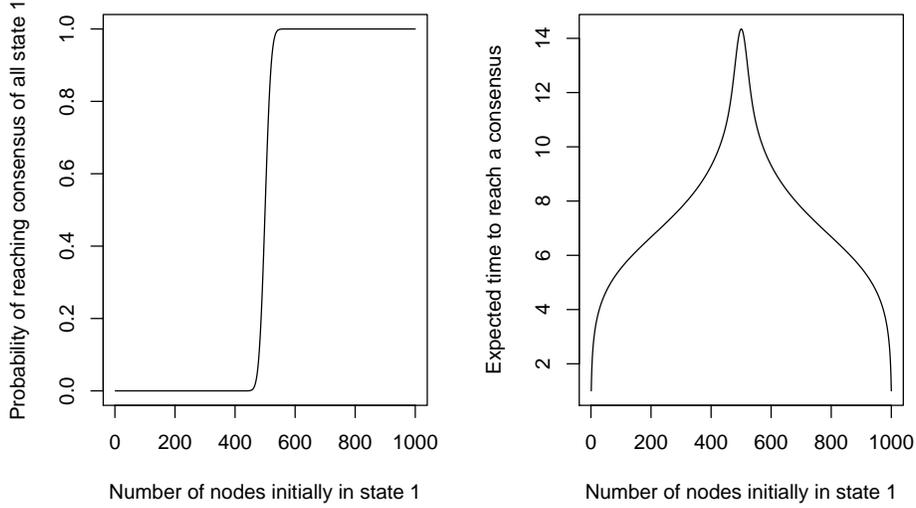, width=\textwidth}
\caption{Exact calculations of the probability to reach consensus of all state 1 and expected time to consensus for 1000 nodes. }\label{fig::exact}
\end{figure}

In this paper we have focused on the behaviour of $p^N$ and $t^N$ as we scale the number of nodes, $N$, to infinity. For a fixed $0<\alpha <1/2$ we set $x= \lfloor \alpha N\rfloor$, Theorem 1 tells us that $p^N(x)$ decays exponentially in $N$ and from Theorem 3 we have $t^N(x)$ grows logarithmically in $N$. In figure \ref{fig::scaling}  we plot $\log(p^N(\lfloor  N/3 \rfloor))$ and $t^N(\lfloor N/3 \rfloor)$ against $N$, the number of nodes. For $p^N$ we have also include a dotted line describing the asymptotic behaviour from Theorem 1. From these plots we can see agreement with the theory even for a small number of nodes. 

\begin{figure}[!t]
\centering
\epsfig{file=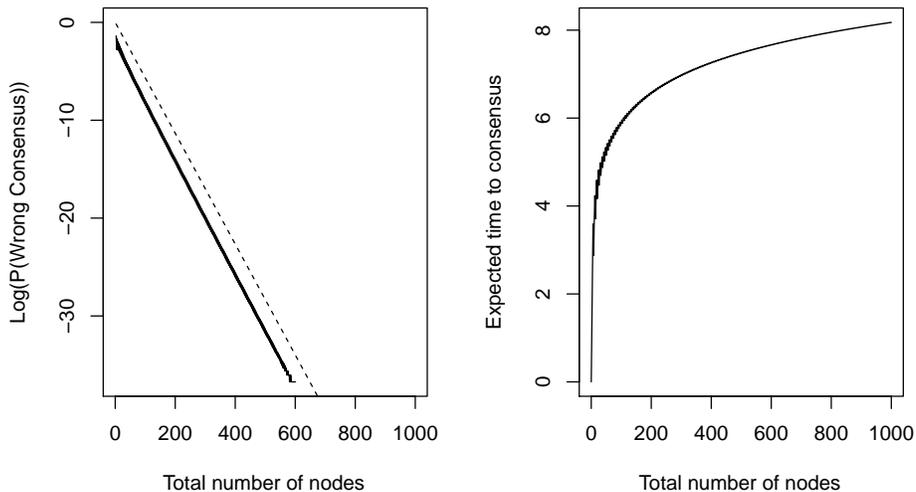, width=\textwidth}
\caption{Scaling behaviour of the probability of wrong consensus and expected time to reach consensus with a 1/3 of nodes initially with state 1. Dashed line in the left hand plot shows asymptotic behaviour from Theorem 1 with $c=1$. }\label{fig::scaling}
\end{figure}

\section{Analysis of the stochastic case} \label{analysis_stoch}

So far we have focused on the class of $(m,d)$ algorithms where the $m$ and $d$ are determined at the start, and are fixed and equal for all nodes. A simple generalization is to allow node heterogeneity, i.e., to allow $m$ and $d$ to be different at each node. In the social learning context, this would model differences among individuals in how they form and modify opinions, and in the decentralized consensus context it could model differences in the connectivity of nodes. The most general version of such a system, in which every node has its own values of the parameters $m$ and $d$, appears analytically intractable. One way to get a tractable approximation would be to assume that nodes belong to a small number of classes, and that the parameters $m$ and $d$ only depend on the class. This still gives rise to a multi-dimensional Markov chain. We consider an even simpler approximation where we allow $m$ and $d$ to be randomly selected each time a node carries out an update. More precisely, each time a node carries out an update, it selects a pair $(m,d)$ according to the distribution of a random vector $(M,D)$; selections are independent and identically distributed (iid) across nodes and over time. In this case, the system can be described by a one-dimensional Markov chain as before, since the number of nodes with value $1$ is an adequate state descriptor. We will assume that the random vector $(M,D)$ satisfies the inequalities $2D>M\ge D$ almost surely, i.e., that a state change always requires at least a simple majority of the polled nodes to have a different state. The assumption is intuitive and does not appear overly restrictive from a practical point of view.

The analysis which we carried out for the deterministic $(m,d)$ algorithms  (Theorems \ref{hitprob} and \ref{hittime}) carries through with only minor changes. Specifically, we need to re-examine the definition of the function $g$. In the deterministic case,  $g$ was defined in terms of the probabilities of the number of nodes in state $1$ increasing or decreasing by $1$ at the next transition. In the stochastic setting, this is still the case, but we now have conditional probabilities depending on the realization of the random vector $(M,D)$. Hence, we need to take an expectation with respect to the distribution of $(M,D)$ to obtain the corresponding unconditional probabilities.

We start by defining the following functions:
$$p_1(m,d,x)=\mathbb{P}( Z_{m,x}\leq(m-d))$$
and
$$p_2(m,d,x) =\mathbb{P}( Z_{m,x}\geq d)$$
where $Z_{m,x}$ denotes a random variable with the Bin$(m,x)$ distribution as before, but we now make the dependence on $m$ explicit in the notation. Analogous to $g$, we now define
$$\tilde g(x)=\frac{x \mathbb{E}_{(M,D)}( p_1(M,D,x))}{(1-x) \mathbb{E}_{(M,D)}( p_2(M,D,x))}.$$
With this definition in place, we can now restate our results as follows:
\begin{theorem} \label{random}
Consider the $(M,D)$ algorithm where, at each update step, each node uses an independent sample of the random variable $(M,D)$ to decide on the number of nodes to poll and to set the decision threshold for changing its state. We assume that $2D>M\geq D$ almost surely. For this algorithm, the probability of reaching consensus on the incorrect value satisfies
$$\limsup_{N \rightarrow \infty} \frac{1}{N} \log(h_N(\alpha)) \leq - \int_{\alpha}^{\frac{1}{2}} \log \tilde g(x) dx. $$
 The time to reach full consensus is $O(\log n)$ for all initial conditions, as in the deterministic case.
\end{theorem}

As an example of the application of the results we consider the randomized algorithm which selects the (1,1) algorithm with probability $p$ and the (2,2) algorithm with probability $(1-p)$. So in an update event the node either asks one or two of its neighbors there current state and only changes its state if all people it asks disagree with its current state. So, in this setting we have
$${\tilde g(x)} = \frac{(1- (1-p)x)}{(p+(1-p)x)}.$$

From this we see that for $x\in[0,1/2)$ ${\tilde g(x)}>1$ which means that we have exponential decay of the
error probability in the number of nodes. Moreover, the time to reach consensus is still $O(log(N))$.
In other words, the benefits of polling two nodes instead of one are realised even if this modification
is only implemented  an arbitrary small fixed  fraction of the time not depending on the number of nodes.
We can use the result in Theorem \ref{hitprob} to find the  limit for the error probability from an initial proportion $x$, $h_N(x)$, given by
$$\lim_{N \rightarrow \infty} \frac{1}{N} \log(h_N(x)) \leq - \left( I(1/2) - I(x) \right),$$
where
\begin{align*}
I(x)=&(x-\frac{1}{(1-p)})\log(1-(1-p)x)\\&-(x+\frac{p}{(1-p)})\log(p+(1-p)x).
\end{align*}
In figure \ref{fig:ranprob} we plot  the limit bound on the exponent for various values of $p$.


\begin{figure}[!t]
\centering
\epsfig{file=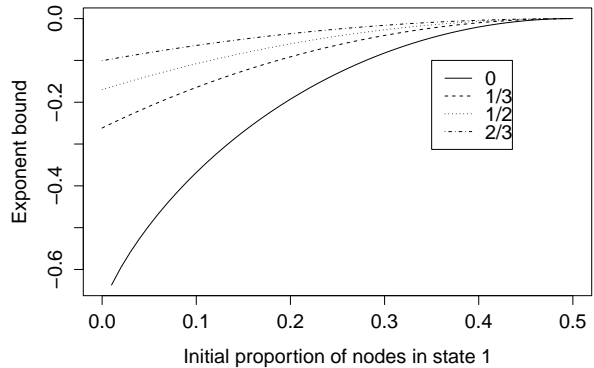, width=0.65\textwidth}
\caption{The upper bound on the exponent of error probability in the limit}\label{fig:ranprob}
\end{figure}

\section{Discussion and directions for future research} \label{sec:discussion}

We considered a number of variants of the classical voter model based on quite general schemes for polling agents and modifying opinions based on majority voting rules. These variants are of potential interest as models of opinion formation and propagation in social networks. However, our results only apply to complete networks. Extending the results to a wide class of network models is a key challenge for future work, and would be of particular relevance to the propagation of opinions in social networks.  Our methods don't immediately extend to other graphs as keeping track of the number of nodes in each state is insufficient, and a Markovian description would require keeping track of the states of all nodes. Thus, the size of the state space is $2^n$ for a general $n$-node graph. Hence, a direct approach seems intractable, and one would seek suitable bounds or approximations. The voter model on general graphs is related to coalescing random walks, and has a martingale associated with this description, which permit easy analysis. There are no such obvious connections for our generalised voter model, which makes this problem challenging on general graphs. Therefore, we believe novel ideas are required.

Our motivation for this work came in part from \cite{milan}, where it was shown that having internal states changed the error probability from a constant to something decaying exponentially in the system size. Specifically, they considered a model with a single internal state, in which a node with value 0 would need to successively contact 2 nodes with value 1 in order to change its value to a 1. By analogy, this motivated us to consider a model in which each node polls two other nodes at each step and only changes value if both of them differ. For this model, it turns out that we get exactly the same error exponent as in the model of \cite{milan}. However, it is not clear whether or how this generalizes to models with multiple internal states. Are there analogues between such models and our class of generalized voter models? Intuitively, we believe that it should be possible to find relations between the two classes of models, and to rigorously establish inequalities between their error probabilities. This is another potential area for future work.

\appendix
\section{Proof of theorems}
 Throughout the proofs we adopt the convention that any summation with a lower index larger than it's upper index is taken to be zero.
\begin{proof}[Proof of Theorem \ref{mmprob}]
Let $X(t)$ denote the number of nodes in state 1 at time $t$. As we are considering the $(m,m)$ algorithm on a complete graph with $N$ nodes, the dynamics of $X$ are governed by a continuous time Markov chain on the state space $\{ 0,1,\ldots, N \}$, with absorbing states at $0$ and $N$, and with transition rates given by
$$ X \textrm{  goes to  } \begin{cases} X+1 \textrm{ with rate }(N-X) \left(\frac{X}{N}\right)^m ,\\
X-1 \textrm{  with rate  } X \left( \frac{N-X}{N}\right)^m .\end{cases} $$
We are interested in the probability of $X$ hitting state $N$ before state $0$, and hence being absorbed in state $N$. 

If the initial system state is $X(0)=i<N/2$, then hitting state $N$ is equivalent to the system reaching consensus on the incorrect (minority) value. Recall from (\ref{hndef_int}) that we defined
$$
h_N(i/N)=\mathbb{P}(\exists\  t \textrm{ such that } X(t) = N | X(0) =i),
$$
as the hitting probability of state $N$ from an initial state with $i$ nodes in state 1. As discussed in Section \ref{analysis_det_prob}, we obtain by conditioning on the first jump that the hitting probabilities $h_N(\frac{i}{N})$, $i=0,1,\ldots,N$, satisfy the recursion (\ref{eq:hitting}).

In order to solve this recursion, we make use of an analogy with electrical resistor networks (see Snell and Doyle \cite{resistor}), as outlined in Section \ref{analysis_det_prob}. Let $R_i$ be the resistance between nodes $i$ and $i-1$ in a series network of $N$ resistors, and let $V_i$ denote the voltage at node $i$, i.e., the node between resistors $R_i$ and $R_{i+1}$. Set $V_0=0$ and $V_N=1$.  If we normalize $R_1=1$, and take
$$
R_{i+1} = \left(\frac{1-i/N}{i/N}\right)^{m-1} R_i,
$$
then the hitting probabilities $h_N(\frac{i}{N})$, $i=1,\ldots,N-1$ satisfy the same equations (\ref{eq:voltage}) as the voltages  $V_i$, $i=1,\ldots,N-1$  in the resistor network described above. We now solve for these voltages.

We obtain from the above recursion for the resistor values that
$$ 
R_{i} =  \prod_{j=1}^{i-1} \left(\frac{N-j}{j}\right)^{m-1} = {N-1\choose i-1}^{(m-1)},
$$
where we use the convention that an empty product is equal to 1. Now, by Ohm's law, the current through the series resistor network is given by $1/(R_1+R_2+\ldots+R_N)$. Consequently, the voltage at node $i$ is 
$V_i=\sum_{k=1}^i R_k/\sum_{k=1}^N R_k$. Since this is the same as the hitting probability $h_N(i/N)$, we obtain that
\begin{equation*}
h_N(i/N) =\frac{ \sum_{j=0}^{i-1} {N-1\choose j}^{(m-1)}}{ \sum_{j=0}^{N-1} {N-1\choose j}^{(m-1)}}.
\end{equation*}
For $\alpha\in [0,1]$ such that  $\alpha N$ is not an integer, we defined $h_N(\alpha)$ in (\ref{hndef_gen}) by linear interpolation. 
Since $h_N(i/N)$ is increasing in $i$, it follows that
\begin{equation} \label{eq:hitprob_alpha}
h_N(\alpha) \leq\frac{ \sum_{i=0}^{\lceil \alpha N \rceil-1} {N\choose i}^{(m-1)}}{ \sum_{i=0}^{N-1} {N\choose i}^{(m-1)}}.
\end{equation}

We shall obtain an upper bound on the numerator and a lower bound on the denominator in the expression above. First observe that, if $i\le \alpha N$, then
$$
{N-1\choose i-1} = \frac{i}{N-i} {N-1\choose i} \le \frac{\alpha}{1-\alpha} {N-1\choose i},
$$
and so
$$
{N-1 \choose \lceil \alpha N \rceil -i} \le \Bigl( \frac{\alpha}{1-\alpha} \Bigr)^{i-1} {N-1 \choose \lceil \alpha N \rceil-1}
$$
for all $i\ge 1$. It follows that
$$
\sum_{i=0}^{\lceil \alpha N \rceil-1} {N-1 \choose i}^{(m-1)} \le {N-1\choose \lceil \alpha N \rceil-1}^{m-1} 
\left( \sum_{i=0}^{\lceil \alpha N \rceil-1} \Bigl( \frac{\alpha}{1-\alpha} \Bigr)^{(m-1)i}\right).
$$
Now, if $\alpha \in (0,1/2)$ as in the statement of the theorem, then
$$
\sum_{i=0}^{\lceil \alpha N \rceil-1} \Bigl( \frac{\alpha}{1-\alpha} \Bigr)^{(m-1)i} \le \sum_{i=0}^{\infty} \Bigl( \frac{\alpha}{1-\alpha} \Bigr)^{(m-1)i} = c<\infty,
$$
for a constant $c$ that does not depend on $N$. Combining this with Stirling's formula, we obtain that
\begin{equation} \label{hitprob_num_ubd}
\sum_{i=0}^{\lceil \alpha N \rceil-1} {N-1\choose i}^{(m-1)} \le \Bigl( \frac{c}{\sqrt{2\pi \alpha(1-\alpha)N}} \Bigr)^{m-1} e^{(N-1)(m-1)H(\alpha)},
\end{equation}
where we recall that $H(\cdot)$ denotes the binary entropy function. Likewise, we obtain using Stirling's formula that
\begin{equation} \label{hitprob_denom_lbd}
\sum_{i=0}^{N-1} {N-1\choose i}^{(m-1)} \ge  {N-1\choose \lfloor N/2 \rfloor}^{m-1} \ge \Bigl( \frac{2}{\pi N} \Bigr)^{m-1} e^{(N-1)(m-1)H(1/2)}.
\end{equation}
Substituting (\ref{hitprob_num_ubd}) and (\ref{hitprob_denom_lbd}) in (\ref{eq:hitprob_alpha}), we obtain that
$$
h_N(\alpha) \le c e^{-(N-1)(m-1)[H(\alpha)-H(1/2)]} = c\exp(-(N-1)(m-1)D(\alpha;1/2)),
$$
where $c>0$ is a constant that may depend on $\alpha$ but does not depend on $N$. This establishes the first claim of the theorem.

In order to establish the second claim about logarithmic equivalence, we need a lower bound on $h_N(\alpha)$ that matches the above upper bound to within a term that is subexponential in $N$. We obtain such a bound from 
\begin{equation*} 
h_N(\alpha) \geq\frac{ \sum_{i=0}^{\lfloor \alpha N \rfloor-1} {N-1\choose i}^{(m-1)}}{ \sum_{i=0}^{N-1} {N-1\choose i}^{(m-1)}}.
\end{equation*}
by obtaining a lower bound on the numerator and an upper bound on the denominator. First, an upper bound on the denominator is as follows:
\begin{eqnarray*}
\sum_{i=0}^{N-1} {N-1\choose i}^{(m-1)} 
&\le& \left( \sum_{i=0}^{N-1} {N-1\choose i} \right)^{m-1} \\
&\le& 2^{(N-1)(m-1)} = \exp((N-1)(m-1)H(1/2)),
\end{eqnarray*}
since $H(1/2)=\log 2$. Next, we can get a lower bound on the numerator by simply replacing the sum with its last term, 
corresponding to $i=\lfloor \alpha N \rfloor-1$. By Stirling's formula, this gives a lower bound which is identical to the upper 
bound in (\ref{hitprob_num_ubd}), up to a constant. Hence, we get
\begin{eqnarray*}
h_N(\alpha) &\ge& c \Bigl( \frac{1}{\sqrt{2\pi \alpha(1-\alpha)N}} \Bigr)^{m-1} e^{(N-1)(m-1)[H(\alpha)-H(1/2)]} \\
&=& cN^{-(m-1)/2} \exp(-(N-1)(m-1)D(\alpha;1/2).
\end{eqnarray*}
Since this agrees with the lower bound on $h_N$ up to a polynomial in $N$, the claimed logarithmic equivalence follows. This completes the proof of the theorem.

\end{proof}

\begin{proof}[Proof of Theorem \ref{hitprob}]

We consider the $(m,d)$ algorithm. Let $X(t)$ denote the number nodes in state 1 at time $t$, and let $Z_x$ denote a  Binomial($m,x$) random variable. Then $X(t)$ evolves as continuous time Markov chain with transition rates
\begin{equation} \label{rates_general}
 X \textrm{ goes to } \begin{cases} X+1 \textrm{ with rate }(N-X) \mathbb{P}(Z_{X/N}\geq d)  ,\\
X-1 \textrm{ with rate }X \mathbb{P}(Z_{X/N}\leq (m-d)) .\end{cases}
\end{equation}
As before, we denote the hitting probability of state $N$ given that we start in state $i$ as
$$h_N(i/N) = \mathbb{P}(X(\infty) =N | X(0)=i).$$
We extend the definition of $h_N$ to $[0,1]$ by linear interpolation, as in (\ref{hndef_gen}).

By conditioning on the first transition, we observe that the hitting probabilities satisfy the recurrence relation (\ref{eq:hitting}),
with the transition probabilities $p_{i,i+1}$ and $p_{i,i-1}$ being proportional to the rates specified above. Hence, as in the
proof of Theorem \ref{mmprob}, we can solve the recurrence using an electrical network analogy. The hitting probabilities
$h_N(i/N)$ satisfy the same equations as the voltages $V_i$ in a series resistor network with $V_0=0$ and $V_N=1$,
where the resistances are related by
$$
R_{i+1}=\frac{p_{i,i-1}}{p_{i,i+1}} R_i = \frac{ i\mathbb{P}(Z_{i/N} \le m-d) }{ (N-i)\mathbb{P}(Z_{i/N}\ge d) } R_i.
$$
Using the definition of the function $g$ in the statement of Theorem \ref{hitprob}, we can rewrite this is
$R_{i+1}= g(i/N) R_i$. Taking $R_1=1$ without loss of generality, and solving for the voltages in this resistor network, we obtain
\begin{equation} \label{eq:hitprob_1}
h_N(i/N) = \frac{ \sum_{j=1}^i R_j }{ \sum_{j=1}^N R_j } = 
\frac{ \sum_{j=1}^i \prod_{k=1}^{j-1} g\bigl( \frac{k}{N} \bigr) }{ \sum_{j=1}^N \prod_{k=1}^{j-1} g\bigl( \frac{k}{N} \bigr) }.
\end{equation}
As usual, we take empty sums to be 0 and empty products to be 1. In order to obtain an upper bound on $h_N(i/N)$, we seek
an upper bound on the numerator and a lower bound on the denominator. We shall bound the numerator from above by $N$ times the largest summand, and the denominator from below by a single summand, corresponding to $i=\lfloor N/2 \rfloor$.

We assumed in the statement of Theorem \ref{hitprob} that $m\ge 2$ and $d>m/2$. Now, 
\begin{equation} \label{eq:gdef}
g(x) = \frac{ x\mathbb{P}(Z_{1-x} \ge d) }{ (1-x)\mathbb{P}(Z_{x}\ge d) } = 
\frac{ \sum_{k=d}^m {m\choose k} (1-x)^k x^{m-k+1} }{ \sum_{k=d}^m {m\choose k} x^k (1-x)^{m-k+1} }.
\end{equation}
Comparing the summands term by term, their ratio is $\bigl(\frac{x}{1-x}\bigr)^{m-2k+1}$ for index $k$. This is bigger than 1
for all $x<1/2$ because $m-2k+1$ is negative or zero for all $k\ge d>m/2$. Therefore, we conclude that $g(x)\ge 1$ for all 
$x\le 1/2$. Using this fact, we see from (\ref{eq:hitprob_1}) that the summands in the numerator are non-decreasing in $j$,
and hence the largest summand corresponds to $j=i$. Noting, also, that the function $h_N$ is non-decreasing, we obtain that
$$
h_N(\alpha) \le \frac{ N\prod_{k=1}^{\lfloor \alpha N\rfloor-1} g\bigl( \frac{k}{N} \bigr)  }{ \prod_{k=1}^{\lfloor N/2 \rfloor}
g\bigl( \frac{k}{N} \bigr)  } = \frac{N}{\prod_{k=\lfloor \alpha N\rfloor}^{\lfloor N/2 \rfloor} g(k/N) }.
$$
Taking logarithms and letting $N$ tend to infinity, we get
$$
\limsup_{N\to \infty} \frac{1}{N}\log h_N(\alpha) \le \limsup_{N\to \infty} \frac{-1}{N}\sum_{k=\lfloor \alpha N\rfloor}^{\lfloor N/2 \rfloor} \log g(k/N).
$$
In order to show that the above sum converges to the Riemann integral of $\int_{\alpha}^{1/2} \log g(x)$, as claimed in the
theorem, it suffices to show that the function $\log g$ is continuous over this compact interval. Now, from (\ref{eq:gdef}), the
function $g$ is a ratio of polynomials that are bounded away from zero on the interval $[\alpha,1/2]$, and the claim follows.
\end{proof}

\begin{proof}[Proof of Theorem \ref{hittime}]

We consider the $(m,d)$ algorithm. Let $X(t)$ denote the number nodes in state 1 at time $t$, and let $Z_x$ denote a  Binomial($m,x$) random variable. Then $X(t)$ evolves as continuous time Markov chain with transition rates given by
(\ref{rates_general}).
We shall bound the time to consensus by introducing a simpler Markov chain whose associated hitting times stochastically dominate those of the consensus process. This new chain can viewed as  a simple random walk with a negative drift  and a reflecting upper boundary. 

Fix $\epsilon \in (0,1)$. We shall define a birth-death Markov chain $Y$ on the integers $\{ 0,1, \ldots, N\}$ by specifying 
the  transition probabilities of the jump chain and the holding times in each state. The definitions involve parameters 
$\beta \in (0,1)$ and $c_1,c_2 >0$ that will be specified later. The transition probabilities of the jump chain are given by
\begin{equation} \label{jump_rates_new}
p_{i,i-1} = 1-p_{i,i+1} = \begin{cases}
\beta, & 1\le i < \frac{(1-\epsilon)N}{2}, \\
\frac{1}{2}, & \frac{(1-\epsilon)N}{2} \le i \le \frac{(1+\epsilon)N}{2}, \\
1-\beta, & \frac{(1+\epsilon)N}{2} < i \le N-1, \\ 
\end{cases}
\end{equation}
while $p_{0,0}=p_{N,N}=1$. The holding times are exponentially distributed, with rates $c_1 i$ in state $i<(1-\epsilon)N/2$, 
$c_1 (N-i)$ for $i>(1+\epsilon)N/2$, and $c_2 N$ for $(1-\epsilon)N/2 \leq i  \leq (1+\epsilon)N/2$. Note that the jump probabilities 
and holding times are symmetric about $N/2$, and that $0$ and $N$ are absorbing states. The jump chain behaves as a symmetric 
random walk in the central section, and as a random walk with drift towards the boundaries in the outer sections.

We shall show that, for a suitable choice of $\beta$, $c_1$ and $c_2$, the time to absorption of the consensus process $X(t)$ 
is stochastically dominated by that for $Y_{\epsilon}(t)$. By explicitly bounding the latter, we shall obtain a bound on the time
to reach consensus. Recall that $Z_x$ denotes a Binomial($m,x$) random variable. We take
\begin{equation} \label{rates_def}
c_1=\mathbb{P}(Z_{\frac{1-\epsilon}{2}} \le m-d), \; c_2 = \mathbb{P}(Z_{\frac{1-\epsilon}{2}} \ge d), \;
\beta = \frac{ (1-\epsilon) c_2 }{ (1-\epsilon)c_2 + (1+\epsilon)c_1 }.
\end{equation}
Define $\tilde X(t) = \min \{ X(t), N-X(t) \}$, and $\tilde Y(t) = \min \{ Y(t), N-Y(t) \}$. Then $\tilde X$ 
and $\tilde Y$ are Markov processes on $\{ 0,1,\ldots,\lfloor N/2 \rfloor \}$, with the same jump rates and holding times
as $X$ and $Y$ respectively, except at state $\lfloor N/2 \rfloor$ where there is no upward jump, and where the holding
time is suitably modified to account for the censored jumps. Moreover, the modified processes $\tilde X$ and $\tilde Y$
have a unique absorbing state at 0, and the same time to absorption as $X$ and $Y$ respectively.

The claim of the theorem will be immediate from the following two lemmas. In fact, they establish that $\tau^0_N(x)$ is no
bigger than $\tilde \tau^0_N(x)$ (defined below), and that the latter is $O(\log N)$. To show that $\tau^0_N(x)$ is in fact
$\Theta(\log N)$, observe that for consensus to be reached, either every node in the minority state has to have updated its
opinion at least once, or every node in the majority state needs to have done so. As each node updates its opinion after
independent Exp(1) times, this involved the maximum of $O(N)$ independent Exp(1) random variables, which has mean
of order $\log N$.
\end{proof}
\begin{lemma} \label{lem:stoch_domination}
For given initial conditions $\tilde X(0) \le \tilde Y(0)$, the stochastic processes $\tilde X(t)$ and $\tilde Y(t)$ 
can be coupled in such a way that $\tilde X(t) \le \tilde Y(t)$ for all $t\ge 0$. In particular, the process $\tilde X$
first falls below level $\lfloor\alpha N \rfloor$, and likewise hits the absorbing state 0, no later than $\tilde Y$ does the same.
\end{lemma}

\begin{lemma} \label{lem:hit_time_bound}
Let $x\in (0,1/2)$ and $\alpha \in (0,x)$ be given. Fix $\epsilon \in (0,1-2x)$, and let the processes $Y$ and $\tilde Y$
be defined as above. Let
\begin{eqnarray*}
\tilde \tau^{\alpha}_N(x) &=& \mathbb{E}_{\lfloor xN \rfloor} (\inf \{t>0: \tilde Y(t)  \le \alpha N), \\
\tilde \tau^0_N(x) &=& \mathbb{E}_{\lfloor xN \rfloor} (\inf \{t>0: \tilde Y(t) =0),
\end{eqnarray*}
where the subscript on the expectation denotes the initial state $\tilde Y(0)$. Then, $\tilde \tau^{\alpha}_N(x)=O(1)$
and $\tilde \tau^0_N(x)=O(\log N)$.
\end{lemma}

\begin{proof}[Proof of Lemma \ref{lem:stoch_domination}]

Let $i< N/2$. The jump probability from $i$ to $i-1$ for the Markov process $\tilde X(t)$ is the same as for $X(t)$ and
is given by
$$
\frac{ i\mathbb{P}(Z_{i/N} \le m-d) }{  i\mathbb{P}(Z_{i/N} \le m-d) +  (N-i)\mathbb{P}(Z_{i/N} \ge d)}.
$$
If $i<(1-\epsilon)N/2$, then this quantity is no smaller than the corresponding jump probability $\beta$ for $\tilde Y(t)$, defined in (\ref{rates_def}); this is because $\mathbb{P}(Z_{i/N} \le m-d)$ is decreasing in $i$ while $\mathbb{P}(Z_{i/N} \ge d)$ is 
increasing. Also, if $d$ is bigger than $m/2$, as we assume, then, for all $i$ between $(1-\epsilon)N/2$ and $N/2$, the jump 
probability of $X$ from $i$ to $i-1$ is no smaller than a half, since 
\begin{equation} \label{drift_ineq_1}
\mathbb{P}(Z_{i/N} \le m-d)  \ge \mathbb{P}(Z_{(N-i)/N} \le m-d) = \mathbb{P}(Z_{i/N} \ge d).
\end{equation} 

Next, we compare the holding times in different states for the two processes. The rate of moving out of state 
$i$ for the $\tilde X$ process is 
$$
i\mathbb{P}(Z_{i/N}\le m-d) + (N-i)\mathbb{P}(Z_{i/N} \ge d),
$$
whereas for the $\tilde Y$ process it is 
\begin{eqnarray*}
i\mathbb{P}(Z_{(1-\epsilon)N/2} \le m-d), && i<(1-\epsilon)N/2 \\
N\mathbb{P}(Z_{(1-\epsilon)N/2} \ge d), && (1-\epsilon)N/2\le i <\lfloor N/2 \rfloor. 
\end{eqnarray*}
As $\mathbb{P}(Z_{i/N}\le m-d)$ 
is decreasing in $i$, it is clear that this rate is greater for the $\tilde X$ process if $i<(1-\epsilon)N/2$. It can be seen
using (\ref{drift_ineq_1}) that it is also greater for $(1-\epsilon)N/2 \le i < \lfloor N/2 \rfloor$.

Thus, we have shown that the jump probability from $i$ to $i-1$  for $\tilde X$ is no smaller than that for $\tilde Y$,
and that the holding time in each state is no greater (in the standard stochastic order). The existence of the claimed coupling
follows from these facts. Indeed, such a coupling can be constructed by letting the processes involve independently when
$\tilde X(t) \neq \tilde Y(t)$, but, when they are equal, sampling the residual holding times and the subsequent jumps
jointly to respect the desired ordering.
\end{proof}

\begin{proof}[Proof of Lemma \ref{lem:hit_time_bound}]

Define $\hat Y(t) = \min \{ \tilde Y(t), \lceil \frac{(1-\epsilon)N}{2} \rceil \}$. Then, $\hat Y(t)$ is a semi-Markov process.
The associated jump chain is Markovian with the same transition probabilities as the $\tilde Y(t)$ process, except that the only
possible transition from $\lceil \frac{(1-\epsilon)N}{2} \rceil $ is to $\lceil \frac{(1-\epsilon)N}{2} \rceil -1$. The holding times in 
all states other $\lceil \frac{(1-\epsilon)N}{2} \rceil$ are exponential with the same rates as in the process $Y(t)$, but the 
holding times in $\lceil \frac{(1-\epsilon)N}{2} \rceil$ have the distribution of the exit time of the process $Y(t)$ from the 
interval $\{ \lceil \frac{(1-\epsilon)N}{2} \rceil, \ldots, \lceil \frac{(1+\epsilon)N}{2} \rceil \}$ in each visit.

Let $t_j$ denote the mean time spent by the process $\hat Y(t)$ in state $j$ during each visit. Then $t_j = 1/(jc_1)$ for 
$j<\lceil \frac{(1-\epsilon)N}{2} \rceil$. In order to compute $t_{\lceil (1-\epsilon)N/2 \rceil}$, we first note that, by well known 
results for the symmetric random walk, the mean number of steps for $Y(t)$ to exit the interval 
$\{ \lceil \frac{(1-\epsilon)N}{2} \rceil \}, \ldots, \lceil \frac{(1+\epsilon)N}{2} \rceil \}$ after entering it at the boundary
is $\lceil \epsilon N \rceil$. Moreover, the mean holding time in each state in this interval is $1/(Nc_2)$ by definition. Hence, 
the mean exit time for $Y(t)$ from this interval, which is also the mean holding time for $\hat Y(t)$ in 
$\lceil \frac{(1-\epsilon)N}{2} \rceil$, is $\epsilon/c_2$. This is a constant that does not depend on $N$.

Now, $\tilde \tau_N^0(x)$ is the mean time for $\tilde Y$, and hence $\hat Y$, to hit 0. By decomposing this into the
number of visits to each intermediate state, and the expected time in each state during each visit, we can write
\begin{equation} \label{tau0_1}
\tilde \tau_N^0(x) = \sum_{j=1}^{\lceil (1-\epsilon)N/2 \rceil} f_{\lfloor xN \rfloor,j} n_j t_j,
\end{equation}
where $f_{ij}$ denotes the probability that $\hat Y$, started in state $i$, hits state $j$ before 0, and $n_j$ denotes
the mean number of visits to state $j$ conditional on ever visiting it. If we let $f_{jj}$ denote the return probability to
state $j$ before hitting 0, then the number of visits to $j$ is geometrically distributed with parameter $f_{jj}$ by the Markov
property, and so $n_j = 1/(1-f_{jj})$. Using well known results for the gambler's ruin problem, we have
$$
f_{ij}=\begin{cases}
\frac{ (\frac{\beta}{1-\beta})^i-1 }{ (\frac{\beta}{1-\beta})^j-1 }, & i<j \cr
1, & i>j,
\end{cases}
$$
where we recall that $\beta>1/2$ is the transition probability from $i$ to $i-1$ for the jump chain associated with $\hat Y(t)$,
which is a biased random walk with reflection at $\lceil (1-\epsilon) N/2 \rceil$. By conditioning on the first step, we also have
$$
f_{jj}=\begin{cases}
1-\beta+\beta f_{j-1,j}, & j\neq \lceil (1-\epsilon) N/2 \rceil, \cr
f_{j-1,j}, & j = \lceil (1-\epsilon) N/2 \rceil.
\end{cases}
$$
Solving for $f_{jj}$ and substituting in $n_j=1/(1-f_{jj})$, we obtain after some tedious calculations that
$$
n_j = \frac{1-(\frac{1-\beta}{\beta})^j}{2\beta-1} \le \frac{1}{2\beta-1}.
$$
But $\beta$ is a constant that does not depend on $N$, and is strictly bigger than a half under the assumption that
$d>m/2$ made in the statement of the theorem. Hence, we see that the expected number of returns to any state
is bounded uniformly by a finite constant.

Now, substituting the results obtained above for $n_j$ and $t_j$ in (\ref{tau0_1}), and noting that $f_{\lfloor xN \rfloor,j} \le 1$
for all $j$ as it is a probability, we obtain that
$$
\tilde \tau^0_N(x) \le \frac{1}{2\beta-1} \Bigl( \sum_{j=1}^{\lceil (1-\epsilon)N/2 \rceil-1} \frac{1}{jc_1} + \frac{\epsilon}{c_2}
\Bigr) = O(\log N).
$$
In fact, a slightly more careful analysis shows that the time to absorption is $O(\log \tilde Y(0))$, which is a tighter bound
if the initial condition grows slower than a fraction of $N$, i.e., if the population is already close to consensus to start with.

The derivation of the upper bound on $\tilde \tau^{\alpha}_N(x)$ is very similar, except that we get an analogue of (\ref{tau0_1})
where the sum only runs over $j\ge \alpha N$. Moreover, the terms in the sum have the interpretation of hitting probabilities and
number of returns before crossing the level $\alpha N$, which is bounded by the hitting probabilities and number of returns before
hitting 0. The details are omitted.
\end{proof}

\begin{proof}[Proof of Theorem \ref{random}]
Let $X(t)$ be the number nodes in state 1 at time $t$. The dynamics of $X$ are governed by the continuous time Markov chain. Before stating the transition rates we define the two functions $p_1$ and $p_2$ by
$$p_1(m,d,x)=\mathbb{P}( Z_{m,x}\leq(m-d))$$
and
$$p_2(m,d,x) =\mathbb{P}( Z_{m,x}\geq d)$$
where $Z_{m,x}$ denotes a random variable with the Bin$(m,x)$ distribution as before, but we now make the dependence on $m$ explicit in the notation.
Then the transition rates of the Markov Chain are
$$ X \textrm{goes to } \begin{cases} X+1 \textrm{ with rate }(N-X)  \mathbb{E}_{(M,D)}(p_2(M,D,X/N)),\\
X-1 \textrm{ with rate }X \mathbb{E}_{(M,D)}(p_1(M,D,X/N)) ,\end{cases} $$
where $\mathbb{E}_{(M,D)}$ is expectation taken with respect to the random choice of $(M,D)$, the algorithm to be used. Finding the hitting probability and the times to consensus follows exactly the same process as for the deterministic case with the minor changes. Therefore we do not include the proofs here.  
\end{proof}

\end{document}